\documentclass[a4paper,12pt]{article}
\usepackage{geometry,latexsym,amssymb}
\usepackage[latin5]{inputenc}
\usepackage{enumerate}
\usepackage{enumitem}
\usepackage[usenames]{color}
\usepackage{graphicx}
\usepackage{enumerate}
\usepackage{rotating}
\usepackage{multirow}
\usepackage{lscape}
\usepackage{cite}
\usepackage{amsthm}
\usepackage{caption}
\usepackage{tikz}
\usetikzlibrary{matrix,arrows}
\usepackage{amsmath,amscd}
\usepackage{calc}
\usepackage{tabularx}
\usepackage{ifthen}
\usepackage{eucal}
\usepackage{amssymb,amscd,amsthm, verbatim,amsmath,color,fancyhdr, mathrsfs}
\usepackage{graphicx}
\usepackage{turnstile}
\usepackage[colorinlistoftodos]{todonotes}

\newtheorem{thm}{Theorem}[section]
\newtheorem{prop}[thm]{Proposition}

\newtheorem{cor}[thm]{Corollary}

\newtheorem{ex}[thm]{Example}

\begin{document}

\begin{center}
{\Large \textbf{A combinatorial interpretation of Mahonian numbers of type $B$}} \vspace*{0.5cm}
\end{center}

\vspace*{0.3cm}
\begin{center}
Hasan Arslan$^{a,1}$,\\
$^{a}${\small {\textit{Department of  Mathematics, Faculty of Science, Erciyes University, 38039, Kayseri, Turkey}}}\\
{\small {\textit{ $^{1}$hasanarslan@erciyes.edu.tr}}}
\end{center}

\begin{abstract}
In this paper, we first introduce the number of signed permutations with exactly $k$ inversions, which is denoted by $i_B(n,k)$ and called \textit{Mahonian numbers of type $B$}. Then we provide a recurrence relation for the Mahonian numbers $i_B(n,k)$. In addition, we give an explicit recursive description for the summation of inversions of all permutations in the hyperoctahedral group $B_n$, denoted by $\mathcal{B}_n$. Furthermore, we enumerate the total number of inversions in permutations in the hyperoctahedral group $B_n$ concretely with the help of an inversion statistic and a backward permutation concepts on $B_n$.  

%We investigate their some combinatorial properties. Furthermore, we give two constructive  proofs of the strong $q$-log-concavity of the $q$-Mahonian numbers of type $B$ in $k$ and $n$, respectively. In particular for $q = 1$, we obtain two constructive proofs of the log-concavity of the Mahonian numbers of type $B$ in $k$ and $n$, respectively.

\end{abstract}


\\
%*Corresponding Author: Hasan Arslan

\section{Introduction}
The main aim of this paper is to introduce Mahonian numbers of type $B$, which is defined to be the number of permutations having $k$ inversions and is denoted by $i_B(n,k)$. We will then investigate some combinatorial properties of these numbers by means of the method depending more on an inversion statistic defined on the hyperoctahedral group $B_n$. Of particular interest here are that $i_B(n,k)$ satisfies the recurrence for all $n\geq 2$
\begin{equation*}
i_B(n,k)=\sum_{i=max\{0,k-2n+1\}}^{min\{k,(n-1)^2\}}i_B(n-1,i)
\end{equation*}
and that the total number of inversions in all permutations in the group $B_n$ for all $n\geq 2$ is equal to $2^{n-1}n^2n!$.

Let $\mathbb{Z}$ and $\mathbb{N}$ be the sets of integers and non-negative integers, respectively. We will interested in three sets of integers, that is, $[n]:=\{1,2, \cdots, n\}$,~ $[m,n]:=\{m,m+1, \cdots, n\}$ and $\langle n \rangle:=\{-n,\cdots,-2,-1,1,2, \cdots, n\}$. The group of all the permutations of $[n]$ is the symmetric group $S_n$. Let $\beta=\beta_1 \cdots \beta_n$ be a one-line presentation of $\beta= \bigl(\begin{smallmatrix}
    1 & 2 &  \cdots  & n \\
    \beta_1 & \beta_2 & \cdots   &  \beta_n
  \end{smallmatrix}\bigr) \in S_n$ to mean that $ \beta_i= \beta(i)$ for all $i=1, \cdots, n$. As is well-known from \cite{Stanley2011}, the inversion number, the descent set and the major index of $\beta$ are respectively defined in the following way:
\begin{align*}
inv (\beta)=&| \{(i,j) \in [n] \times [n] ~:~ i<j~and~\beta_i > \beta _{j}\} | \\
Des (\beta)=&\{i \in [n-1] ~:~ \beta_i > \beta _{i+1}\} \\
maj (\beta)=& \sum_{i \in Des (\beta)}i.
\end{align*}
Let $I(n,k):=\{\beta \in S_n~:~inv(\beta)=k\}$ for any two integers $n\geq 1$ and $k \geq 0$ and let $ i(n,k)=|I(n,k)|$. More precisely, $i(n,k)$ is the number of permutations having $k$ inversions and is called as \textit{the classical Mahonian numbers}. Actually, $i(n,k)$ counts the number of different ways of distributing "$k$" balls among "$n-1$" boxes such that the $j$th box contains at most "$n-j$" balls. Denote by $\mathcal{A}_n$ the sum of inversions of all elements $\beta$ in $S_n$, that is, $\mathcal{A}_n=\sum_{\beta \in S_n}inv(\beta)$. It is well-known from \cite{Indong2008} that $\mathcal{A}_n=\frac{n!}{2}\binom{n}{2}$ for $n\geq1$ and it has a recurrence relation for all $n \geq 2$
$$\mathcal{A}_n=\frac{n!(n-1)}{2}+n\mathcal{A}_{n-1}$$
with the initial condition $\mathcal{A}_1=0$.

MacMahon proved in \cite{MacMahon1915} that the inversion statistic "inv" is equi-distributed with the major index "maj" over the symmetric group $S_n$, that is, 
\begin{equation}\label{MacMahon1}
\sum_{\beta \in S_n}q^{inv(\beta)}=\sum_{\beta \in S_n}q^{maj(\beta)}=\prod_{i=1}^n \frac{1-q^i}{1-q}
\end{equation}
where $q$ is an indeterminate.
Considering that the element with the maximum number of inversions in $S_n$ is $\beta= \bigl(\begin{smallmatrix}
    1 & 2 &  \cdots  & n \\
   n & n-1 & \cdots   & 1
\end{smallmatrix}\bigr)$ and $inv(\beta)=\binom{n}{2}$, the equation (\ref{MacMahon1}) can be expressed as
\begin{equation}\label{MacMahon2}
\sum_{k=0}^{\binom{n}{2}} i(n,k) q^k=(1+q)(1+q+q^2)\cdots(1+q+q^2+\cdots+q^{n-1}).
\end{equation}
Furthermore, $i(n,k)$ is equal to zero when $k<0$ or $k>\binom{n}{2}$. As a matter of fact, the classical Mahonian numbers coincide with the coefficients of the terms of the polynomial given in the equation (\ref{MacMahon2}). Due to \cite{Indong2008}, the classical Mahonian numbers have the symmetry property and longitudinal recurrence relation given in the following theorem.

\begin{thm}[Indong and Peralta, \cite{Indong2008}]\label{sym}
For all $n>0$ and $0\leq k \leq \binom{n}{2}$, we have
$$i(n,k)=i(n, \binom{n}{2}-k)$$
and
$$i(n,k)=\sum_{j=max\{0,k-n+1\}}^{min\{k,\binom{n-1}{2}\}}i(n-1,j).$$
\end{thm}

One can create Table \ref{tab:table}, which is known as \textit{the Mahonian triangle} and appears in Sloane \cite{Sloane} as A008302, with the help of the equation (\ref{MacMahon2}).

\begin{table}[h!] 
\caption{The Mahonian triangle}
	\label{tab:table}		
	\centering
	\scalebox{0.9}{
			\begin{tabular}{|c||c|c|c|c|c|c|c|c|c|c|c|c|c|c|c|c|c|}
					\hline
					$n \backslash k$ & $\mathcal{A}_n$&0 & 1 &2&3& 4& 5&6& 7&8&9&10&11&12&13&14&15\\
					\hline \hline
					1&0&1&&&&&&&&&&&&&&&\\
					\hline
					2&1&1&1&&&&&&&&&&&&&&\\
					\hline
					3&9&1&2&2&1&&&&&&&&&&&&\\
					\hline
					4&72&1&3&5&6&5&3&1&&&&&&&&&\\
					\hline
					5&600&1&4&9&15&20&22&20&15&9&4&1&&&&&\\
					\hline
					6&5400&1&5&14&29&49&71&90&101&101&90&71&49&29&14&5&1\\
				          \hline  \hline
					
			\end{tabular}}
\end{table}

Ghemit and Ahmia in \cite{Ghemit2023} introduced a $q$-analogue of the classical Mahonian numbers. Moreover, they provided lattice path and tilling interpretations of $q$-Mahonian numbers and then investigated some combinatorial properties of $q$-Mahonian numbers such as $q$-log-concavity in $k$ and $n$. 

Now let $S_0(x)=1+x^3+5x^4+22x^5+90x^6+359x^7+1415x^8+\cdots$ represent the generating function for the Mahonian numbers $i(n,n)$ and let $C(x)=\frac{1-\sqrt{1-4x}}{2x}$ denote the generating function for the Catalan numbers, $C_n=\frac{1}{n+1}\binom{2n}{n}$. Let $S_j(x)$ be the generating function for the Mahonian numbers $i(n,n-j)$ for a fixed non-negative integer $j$, namely 
\begin{equation*}
S_j(x)=\sum_{n \geq 0} i(n,n-j)x^n.
\end{equation*}
In \cite{claesson2023}, the authors proved that the generating function $S_j(x)$ for the Mahonian numbers $i(n,n-j)$ is precisely equal to 
\begin{equation}\label{gf}
(xC(x))^j S_0(x).
\end{equation}

\section{Preliminaries and Notation}

We start this section by recalling the definition of the group $B_n$, which is also known as a \textit{hyperoctahedral group}. Any element $w\in B_{n}$ acts as a signed permutation on the set $\langle n \rangle$ such that $w(-i)=-w(i)$ for each $i \in [n]$. The group $B_n$ has the canonical set of generators $S=\{t_1,s_1,\cdots, s_{n-1}\}$, and it has the following Dynkin diagram with respect to the set $S$:
\[
\begin{tikzpicture}
[
roundednode/.style={circle, draw=white!60, fill=white!5, very thick, minimum size=6mm},
roundnode/.style={circle, draw=black!50, fill=white!5, very thick, minimum size=6mm},
squarednode/.style={circle, draw=black!60, fill=white!5, very thick, text width=0.005mm, minimum size=6mm},
]
%Nodes
\node[roundnode, font=\tiny]      (maintopic){};
\node[roundednode]  [left=of maintopic]{$B_n$ :};
\node[roundnode]        (uppercircle)       [right=of maintopic]{};
\node[roundnode]      (rightsquare)       [right=of uppercircle]{};
\node[roundnode]        (lowercircle)       [right=of rightsquare]{};
\node[roundnode]        (zcircle)       [right=of lowercircle] {};
\node[node distance= 0.3cm, below = of maintopic]{$t_1$};
\node [label=$\frac{-1}{\sqrt{2}}$] (C) at (0.8,0) {};
\node[node distance= 0.3cm, below = of uppercircle]{$s_1$};
 \node[node distance= 0.3cm, below = of rightsquare]{$s_2$};
 \node[node distance= 0.3cm, below = of lowercircle]{$s_{n-2}$};
 \node[node distance= 0.3cm, below = of zcircle]{$s_{n-1}$};
%Lines
\draw [-,double  distance=2pt,thick](maintopic) -- (uppercircle);
\draw[-] (uppercircle) -- (rightsquare);
\draw[densely dashed,-] (rightsquare)-- (lowercircle);
\draw[-] (lowercircle)-- (zcircle);
\draw[-](lowercircle)-- (zcircle);
\end{tikzpicture}\\
\]

Let $t_{i+1}:=s_{i}t_{i}s_{i}$ for each $ i=1,\cdots, n-1$. The group $B_n$ is a semidirect product of $S_n$ and $\mathcal{T}_{n} $, written $B_n= S_n \rtimes \mathcal{T}_{n} $, where $S_n$ is the symmetric group generated by $\{ s_1, \cdots, s_{n-1}\}$ and $\mathcal{T}_{n}$ is a normal subgroup of $B_n$ generated by $\{ t_1, \cdots, t_n \}$. The cardinality of the group $B_{n}$ is clearly $2^{n} n!$. Note here that $s_i~~(i=1,\cdots,n-1)$ is a transposition identified with 
$$s_i(j)=
\begin{cases}
		i+1,  &  \textrm{if}~ j=i; \\
  	i,  &  \textrm{if}~ j = i+1;\\
		j,  &  \textrm{otherwise.}
\end{cases}
$$ and $t_i~~(i=1,\cdots,n)$ is a permutation defined by $$t_i(j)=
\begin{cases}
		j,  & \textrm{if} ~ j \neq i; \\
		-i,  & \textrm{if} ~ j=i.
\end{cases}
$$
The group $B_n$ has a nice representation given by the following form (the case $m=2$ in \cite{Davies1974}):
\begin{align*}
    B_n=&\langle  s_1,\cdots, s_{n-1}, t_1,\cdots, t_{n} : s_i^2=(s_i s_{i+1})^3=(s_i s_j)^2=e, |i-j|>1;\\ & t_{i}^2=(t_1s_1)^4=e, t_i t_j=t_j t_i, s_i t_i s_i=t_{i+1}, s_i t_j=t_j s_i,~j\neq i,i+1 \rangle
\end{align*}
where $e$ denotes the identity element of $B_n$. Any element $w \in B_n$ can be uniquely written in the form
\[
w = \bigl(\begin{smallmatrix}
	1 & 2 &  ~~\cdots  & n \\
	(-1)^{r_1}\beta_1 & ~~(-1)^{r_2}\beta_2 & ~~\cdots   & ~~ (-1)^{r_n}\beta_n
\end{smallmatrix}\bigr)=\beta \prod_{k=1}^n t_{k}^{r_k}\in B_n
\]
where $r_i \in \{0,1\}$ and we write $\beta = \bigl(\begin{smallmatrix}
	1 & 2 &  \cdots  & n \\
	\beta_1 & \beta_2 & \cdots   &  \beta_n
\end{smallmatrix}\bigr) \in S_n$.  For further information about the hyperoctahedral group $B_n$, see \cite{bjorner2005}.

Following \cite{flag2001} we let, $\gamma_0:=t_1$ and $\gamma_i:=s_is_{i-1}\cdots s_1 t_1 \in B_n$ for all $i=1, \cdots, n-1$. Thus, the collection $\{\gamma_0, \gamma_1, \cdots, \gamma_{n-1}\}$ is a different set of generators for $B_n$ and any $w\in B_n$ has a unique expression 
\begin{equation*}
w=\gamma_{n-1}^{k_{n-1}}\cdots \gamma_{2}^{k_{2}}\gamma_{1}^{k_{1}}\gamma_{0}^{k_{0}}
\end{equation*}
with $0\leq k_i \leq 2i+1$ for all $i=0,\cdots,n-1$. \textit{Flag-major index} was defined for the group $B_n$ as follows (see \cite{flag2001}): Let $w\in B_n$. Then
\begin{equation*}
fmaj(w)=\sum_{i=0}^{n-1}k_i.
\end{equation*}
It is well-known from \cite{flag2001} that the flag-major index  is Mahonian, that is, 
\begin{equation*}
\sum_{w \in B_n}q^{fmaj(w)}=\prod_{i=1}^n \frac{~1-q^{2i}}{1-q}.
\end{equation*}

\textit{The inversion table} of $w \in B_n$ is defined to be the sequence $Inv_B(w) = (inv_1(w) : \cdots : inv_n(w))$, where $inv_i(w)$ is called $i$-inversion of the permutation $w$ for each $i=1,\cdots, n$. We denote the sum of $i$-inversions of the permutation $w \in B_n$ by $inv_B(w)$, namely $inv_B(w):=\sum_{i=1}^n inv_i(w)$. Note that $inv_B$ agrees with the length function defined with respect to the generating set $S$ on $B_n$. The inversion table of $w$ can be easily constructed with the help of the following theorem:

\begin{thm}\label{3}
For $w=\beta \prod_{k=1}^n t_{k}^{r_k} \in B_n$, we have
\begin{equation}\label{33}
inv_i(w)=r_{n+1-i}+2.\mid \{(j,n+1-i) : j<n+1-i, ~~\beta_j<\beta_{n+1-i}, r_{n+1-i}\neq 0\} \mid+inv_i(\beta)
\end{equation}
for all $i=1,\cdots,n$, where $inv_i(\beta)=\mid\{(j,n+1-i) : j<n+1-i,~~ \beta_j>\beta_{n+1-i} \} \mid$ in $S_n$ and $r_{n+1-i} \in \{0,~1\}$.  More precisely, $inv_i(w)=1+2.\mid \{(j,n+1-i) : j<n+1-i, ~~\beta_j<\beta_{n+1-i}\} \mid+inv_i(\beta)$ when $r_{n+1-i}=1$ and   $inv_i(w)=inv_i(\beta)$ when $r_{n+1-i}=0$. 
\end{thm}

Theorem \ref{3} is a special case of Theorem 4.5 in \cite{hasan2022} (the case $m=2$). We here note that $inv_i(w) \in [0,2(n-i)+1]$ for all $i=1,\cdots,n$.

 \begin{ex}
Let $w= \bigl(\begin{smallmatrix}
    1 & ~~2 & ~~~~ 3&~~ 4&~~~5&~~~6 &~~~7&~~8\\
   7&~~ 3 &~~ -2 &~~8&~~ -6&~~ -4&~~-1&~~5
\end{smallmatrix}\bigr)\in B_8.$ Considering the equation (\ref{33}) we obtain the inversion table of $w$ as $Inv_B(w)=(3:7:8:7:0:3:1:0)$, and so $inv_B(w)=29$.
 \end{ex}
If we specifically take $m=2$ in Proposition 4.12 of \cite{hasan2022}, we obtain the following result.
\begin{prop}\label{p1}
Let
\begin{align*}
\mathcal{T}_{2,n}&=\{(a_1 : \cdots : a_n) ~:~ 0 \leq a_i \leq 2(n-i+1)-1,~i=1,\cdots, n\}\\
&=[0,2n-1]\times [0,2n-3]\times \cdots \times [0,3]\times [0,1]. 
\end{align*}
The correspondence $w \mapsto Inv(w)$ gives a bijection between $B_n$ and $\mathcal{T}_{2,n}$.
\end{prop}

\section{A partition of the hyperoctahedral group}

Any two permutations $\sigma_1$ and $\sigma_2$ in $B_n$ are related, written $\sigma_1 \sim \sigma_2$,  if and only if $\sigma_1(n)=\sigma_2(n)$. It can be verified that $ \sim$ is an equivalence relation on $B_n$. The equivalence relation $\sim$ give rises to equivalence classes $C_j=\{ \sigma \in B_n~:~ \sigma (n)=j\}$ of $B_n$, where $j=-n,\cdots,-1,1,\cdots,n$. It is clear that $|C_j|=2^{n-1}(n-1)!$ for all  $j=-n,\cdots,-1,1,\cdots,n$. Thus we have a decomposition $B_n=\biguplus_{j=-n}^n{ C_j}$. Therefore, we have 
\begin{equation}\label{r1}
\mathcal{B}_n=\sum_{j=-n}^n\sum_{\sigma \in C_j}inv_B(\sigma)
\end{equation}
where $\mathcal{B}_n$ represents the total number of inversions of all permutations in $B_n$, namely $\mathcal{B}_n=\sum_{w \in B_n}inv_B(w)$. Let $\sigma=\bigl(\begin{smallmatrix}
	1 & ~~2 &  ~~\cdots &~~n-1 &~ n \\
	\sigma(1) & ~~\sigma(2) & ~~\cdots  & ~~\sigma(n-1) & ~~j
\end{smallmatrix}\bigr)  \in C_j$ and $a_1,\cdots,a_{n-1}$ be an arrangement of elements of $[n]\backslash \{|j|\}$  in increasing order. Now we associate $\sigma$ with the permutation $\tau$, which is defined by 
$$\tau:=\bigl(\begin{smallmatrix}
	a_1 & ~~a_2 &  ~~\cdots &~~a_{n-1}  \\
	\sigma(1 )& ~~\sigma(2) & ~~\cdots  & ~~\sigma(n-1)
\end{smallmatrix}\bigr). $$ 
Note that $\tau$ is a signed permutation of $[n]\backslash \{|j|\}$. We will denote by $\bold{P}([n]\backslash \{|j|\})$ the group of all the signed permutation of the set $[n]\backslash \{|j|\}$. In addition, if we define 
$$\sigma_{\tau,j}:=\bigl(\begin{smallmatrix}
	1 & ~~2 &  ~~\cdots &~~n-1 &~~ n \\
	\tau(a_1) & ~~\tau(a_2) & ~~\cdots  & ~~\tau(a_{n-1}) & ~~j
\end{smallmatrix}\bigr)$$
 then we see that  $\sigma=\sigma_{\tau,j}$. Hence, by Theorem \ref{3}, we conclude that
\begin{equation}\label{r2}
inv_B(\sigma)=
\begin{cases}
     n-j+ inv_B(\tau) & j > 0 \\
      n-j-1+ inv_B(\tau) & j < 0 \\
   \end{cases}.
\end{equation}
Based on the equations (\ref{r1}) and (\ref{r2}), we will give a recursive formula for $\mathcal{B}_n$ in the following proposition.

\begin{prop}\label{l1}
We have $\mathcal{B}_1=1$ and for all $n \geq 2$
$$\mathcal{B}_n=2^{n-1}n!(2n-1)+2n\mathcal{B}_{n-1}.$$
\end{prop}

\begin{proof}
Since $B_1=\{e, \bigl(\begin{smallmatrix}
	~1\\
	-1 
\end{smallmatrix}\bigr)  \}$, then it is easy to see that $\mathcal{B}_1=1$. Now suppose $n \geq 2$.\\
\textbf{Case 1}. For each $j=1,\cdots,n$, we obtain
\begin{align*}
\sum_{\sigma \in C_j}inv_B(\sigma)&=\sum_{\tau \in \bold{P}([n]\backslash \{|j|\})}inv_B(\sigma_{\tau,j})\\
&=\sum_{\tau \in \bold{P}([n]\backslash \{|j|\})} \left( ( n-j)+ inv_B(\tau) \right)\\
&=2^{n-1}(n-1)!(n-j)+\sum_{\tau \in \bold{P}([n]\backslash \{|j|\})}  inv_B(\tau) \\
&=2^{n-1}(n-1)!(n-j)+\mathcal{B}_{n-1}.
\end{align*}
Then we get
\begin{align*}
\sum_{j=1}^n \sum_{\sigma \in C_j}inv_B(\sigma)&=\sum_{j=1}^n \left(2^{n-1}(n-1)!(n-j)+\mathcal{B}_{n-1}\right)\\
&=2^{n-2}n!(n-1)+n\mathcal{B}_{n-1}.
\end{align*}
\textbf{Case 2}. For each $j=-n,\cdots,-1$, we deduce
\begin{align*}
\sum_{\sigma \in C_j}inv_B(\sigma)&=\sum_{\tau \in \bold{P}([n]\backslash \{|j|\})}inv_B(\sigma_{\tau,j})\\
&=\sum_{\tau \in \bold{P}([n]\backslash \{|j|\})} \left( ( n-j-1)+ inv_B(\tau) \right)\\
&=2^{n-1}(n-1)!(n-j-1)+\sum_{\tau \in \bold{P}([n]\backslash \{|j|\})}  inv_B(\tau) \\
&=2^{n-1}(n-1)!(n-j-1)+\mathcal{B}_{n-1}.
\end{align*}
Then we obtain
\begin{align*}
\sum_{j=-n}^{-1} \sum_{\sigma \in C_j}inv_B(\sigma)&=\sum_{j=-n}^{-1} \left(2^{n-1}(n-1)!(n-j-1)+\mathcal{B}_{n-1}\right)\\
&=2^{n-2}n!(3n-1)+n\mathcal{B}_{n-1}.
\end{align*}
Considering Case 1 and Case 2 together, we conclude 

\begin{align*}
\mathcal{B}_n&=\sum_{j=-n}^n\sum_{\sigma \in C_j}inv_B(\sigma)\\
&=2^{n-2}n!(n-1)+n\mathcal{B}_{n-1}+2^{n-2}n!(3n-1)+n\mathcal{B}_{n-1}\\
&=2^{n-1}n!(2n-1)+2n\mathcal{B}_{n-1}
\end{align*}
as desired.

\end{proof}

\section{Mahonian numbers of type B}

For any two integers $n\geq 1$ and $k \geq 0$, let $I_B(n,k):=\{w \in B_n~:~inv_B(w)=k)\}$ and $ i_B(n,k)=|I_B(n,k)|$. In other words, $i_B(n,k)$ represents the number of signed permutations having $k$ inversions and in turn is called as \textit{Mahonian numbers of type B}. Hence, we can interpret these Mahonian numbers combinatorially as follows:\\

\textbf{Combinatorial interpretation:} $i_B(n,k)$ counts the number of ways to place "$k$" balls into "$n$" boxes such that the $j$th box contains at most "$2(n-j)+1$" balls.

The element of $B_n$ having maximum number of inversions is unique and it is expressed by the following form:
\[
w_0=  \left(
\begin{matrix}
	~~1 & ~~2 &~~ 3 & \cdots &~~ n-1 & ~~n \\
	-1 & -2 & -3 & \cdots &  -(n-1)  & -n
\end{matrix}
\right).
\]
It can be easily from Theorem \ref{3} seen that $inv_B(w_0)=n^2$. For this reason, we deduce that $i_B(n,k)=0$ when $k<0$ or $k>n^2$. By \cite{hasan2022}, we can write
\begin{equation}\label{MacMahon3}
\sum_{w \in B_n}q^{inv_B(w)}=\prod_{i=1}^n \frac{1-q^{2i}}{1-q}
\end{equation}
and so we say that the inversion statistic "$inv_B$" and the flag-major index "fmaj" are equi-distributed over $B_n$.  We can restate the equation (\ref{MacMahon3}) as
\begin{equation}\label{MacMahon4}
\sum_{k=0}^{n^2} i_B(n,k) q^k=(1+q)(1+q+q^2+q^3)\cdots(1+q+q^2+q^3+\cdots+q^{2n-1}).
\end{equation}
Note that $\mathcal{B}_n=\sum_{k=0}^{n^2}i_B(n,k)k$. Table \ref{tab:table3} can be established by using the equation (\ref {MacMahon4}) and we will call this table as \textit{the Mahonian triangle of type B}. This table also appears in Sloane \cite{Sloane} as A128084.

%\begin{landscape}
\begin{table}[h!] 
\caption{The Mahonian triangle of type B}
	\label{tab:table3}		
	\centering
	\scalebox{0.55}{
			\begin{tabular}{|c||c|c|c|c|c|c|c|c|c|c|c|c|c|c|c|c|c|c|c|c|c|c|c|c|c|c|c|}
					\hline
					$n \backslash k$ &$\mathcal{B}_n$& 0 & 1 &2&3& 4& 5&6& 7&8&9&10&11&12&13&14&15&16&17&18&19&20&21&22&23&24&25\\
					\hline \hline
					1&1&1&1&&&&&&&&&&&&&&&&&&&&&&&&\\
					\hline
					2&16&1&2&2&2&1&&&&&&&&&&&&&&&&&&&&&\\
					\hline
					3&216&1&3&5&7&8&8&7&5&3&1&&&&&&&&&&&&&&&&\\
					\hline
					4&3072&1&4&9&16&24&32&39&44&46&44&39&32&24&16&9&4&1&&&&&&&&&\\
					\hline
					5&48000&1&5&14&30&54&86&125&169&215&259&297&325&340&340&325&297&259&215&169&125&86&54&30&14&5&1\\
				          \hline  \hline
					
			\end{tabular}}
\end{table}
%\end{landscape}

Based on Proposition \ref{p1}, we can immediately obtain the following result.

\begin{cor} \label{c1}
For each $0 \leq k \leq n^{2}$, there exists an element $w$ of $B_n$ such that $inv_B(w)=k$. 
\end{cor}

\begin{thm}\label{first}
For each $0 \leq k \leq n^{2}$, we have $i_B(n,n^{2}-k)=i_B(n,k)$. 
\end{thm}

\begin{proof}
Let $P_1=\{w\in B_n~:~inv_B(w)=k\}$ and $P_2=\{w\in B_n~:~inv_B(w)=n^2-k\}$. By Corollary \ref{c1}, it is clear that $P_1$ and $P_2$ are both nonempty. Thus the mapping $B~:~P_1 \rightarrow P_2,~B(w)=w_{0}w$ is clearly bijective since $inv_B(w_{0}w)=n^2-inv_B(w)$ for all $w \in B_n$. Therefore, we get $|P_1|=|P_2|$ and $i_B(n,n^{2}-k)=i_B(n,k)$.
\end{proof}
As a result of Theorem \ref{first}, we observe that Mahonian numbers of type $B$ are palindromic in $k$.

\begin{thm}\label{second}
For each $0 \leq k \leq n^{2}$, we have the following recurrence relation 
\begin{equation*}
i_B(1,0)=i_B(1,1)=1
\end{equation*}
and for all $n\geq 2$
\begin{equation}\label{rec}
i_B(n,k)=\sum_{i=max\{0,k-2n+1\}}^{min\{k,(n-1)^2\}}i_B(n-1,i).
\end{equation}
\end{thm}

\begin{proof}
It is clear that $i_B(1,0)=i_B(1,1)=1$. Now suppose $n\geq 2$,~ $\sigma \in I_B(n,k)$ and $\sigma \in \textbf{C}_j$.
%\textbf{Case 1.} If $j>0$, there exists a unique $\tau \in \textbf{P}([n]\setminus \{j\})$ such that $\sigma=\sigma_{\tau, j}$ and $inv_B(\sigma_{\tau, j})=n-j+inv_B(\tau)$. If we set $inv_B(\tau)=i$, then we have $0 \leq i \leq (n-1)^{2}$. We find those $i$ such that given $j$, $inv_B(\sigma_{\tau, j})=k$ and $i_B(n,k)$ can be formed by adding $i_B(n-1,i)$ for all values of $i$ that we found. We have $i=k+j-n \leq k$ and $i \leq (n-1)^2$. Thus we get $i \leq min\{k,(n-1)^2\}$. Since $i$ must be nonnegative and $i=k+j-n\geq k-2n+1 $, then we have $i \geq max\{0,k-2n+1\}$.\\
%\textbf{Case 2.} If $j<0$, we can find a unique $\tau \in \textbf{P}([n]\setminus \{j\})$ such that $\sigma=\sigma_{\tau, j}$ and $inv_B(\sigma_{\tau, j})=n-j-1+inv_B(\tau)$. If we take $inv_B(\tau)=i$, then we have $0 \leq i \leq (n-1)^{2}$. We find those $i$ such that given $j$, $inv_B(\sigma_{\tau, j})=k$ and $i_B(n,k)$ can be formed by adding $i_B(n-1,i)$ for all values of $i$ that we found. We have $i=k+j+1-n \leq k$ and $i \leq (n-1)^2$. Thus we get $i \leq min\{k,(n-1)^2\}$. Since $i$ must be nonnegative and $i=k+j+1-n\geq k-2n+1 $, then we have $i \geq max\{0,k-2n+1\}$.
%Thus we complete the proof.
Since $\sigma \in \textbf{C}_j$, then there exists a unique $\tau \in \textbf{P}([n]\setminus \{|j|\})$ such that $\sigma=\sigma_{\tau, j}$. Therefore, we have $inv_B(\sigma_{\tau, j})=\begin{cases} 
      n-j+inv_B(\tau) & j> 0 \\
     n-j-1+inv_B(\tau) & j< 0  
   \end{cases}$ by the equation (\ref{r2}). There always exists $i$ such that $inv_B(\tau)=i$ and $inv_B(\sigma_{\tau, j})=k$ for a given value $j$. Clearly, we have $0 \leq i \leq (n-1)^{2}$. As a result of this fact, $i_B(n,k)$ can be created by summing $i_B(n-1,i)$ for all values of $i$ that we found.\\
\textbf{Case 1.} Let $j>0$. Then we determine two upper bounds for $i$ as $i=k+j-n \leq k$ and $i \leq (n-1)^2$. Thus we obtain $i \leq min\{k,(n-1)^2\}$. Since $i$ must be non-negative and $i=k+j-n\geq k-n+1 $, then we conclude $i \geq max\{0,k-n+1\}$.\\
\textbf{Case 2.} If $j<0$, then we have $i=k+j+1-n \leq k-n$ and $i \leq (n-1)^2$. Thus we deduce $i \leq min\{k-n,(n-1)^2\}$. Since $i$ must be non-negative and $i=k+j+1-n\geq k-2n+1 $, then we have $i \geq max\{0,k-2n+1\}$.

In calculation process of $i_B(n,k)$, we need to find the number of all signed permutations with exactly $k$ inversions. Therefore, we will first  determine the contribution to the desired sum of the permutations such that each of them has $k$ inversions and such that $j>0$. From Case 1, these permutations contribute $i_B(n-1,i)$ to the sum, where $i$ satisfies $max\{0,k-n+1\} \leq i \leq min\{k, (n-1)^2\}$. On the other hand, as can be clearly seen from Case 2, the contribution of the permutations having $k$ inversions and $j<0$ to the sum is $i_B(n-1,i)$, where $i$ ranges over all non-negative integers included in the interval $[max\{0,k-2n+1\},~ min\{k-n, (n-1)^2\}]$. When these two observations are considered together, we deduce that 
$$i_B(n,k)=\sum_{i=max\{0,k-2n+1\}}^{min\{k,(n-1)^2\}}i_B(n-1,i).$$
Thus we complete the proof.
\end{proof}

\begin{ex}
We can illustrate the recurrence relation in the equation (\ref{rec}) as $i_B(4,7)=\sum_{i=0}^{7}i_B(3,i)=44$,~ $i_B(5,11)=\sum_{i=2}^{11}i_B(4,i)=325$ and $i_B(5,17)=\sum_{i=8}^{16}i_B(4,i)=215$ by using Table \ref{tab:table3}.
\end{ex}

Taking into account the equation (\ref{rec}), it is not hard to see that for all $n\geq 2$ 
$$ i_B(n,k)=\sum_{i=0}^{2n-1}i_B(n-1,k-i). $$

Given  a permutation $w=\bigl(\begin{smallmatrix}
	1 & ~~2 &  ~~\cdots &~~n-1 &~ n \\
	w_1 & ~~w_2 & ~~\cdots  & ~~w_{n-1} & ~~ w_n
\end{smallmatrix}\bigr) \in B_n$. The \textit{backward permutation} associated with the permutation $w \in B_n$, denoted by $\overleftarrow{w}$, is defined as
$$\overleftarrow{w}=\bigl(\begin{smallmatrix}
	1 & ~2 &  ~~\cdots  &~~ n-1&~~ n \\
	w_n & ~~w_{n-1} & ~~\cdots   & ~~ w_2& ~~ w_1
\end{smallmatrix}\bigr).$$
Essentially, the backward permutation of a given permutation in $B_n$ is obtained by reversing the elements in its row images. Clearly, any backward permutation  $\overleftarrow{w}$ is also an element of $B_n$.
\begin{thm}\label{t2}
For $n>1$, we have
\begin{equation*}
\mathcal{B}_n=\sum_{w \in B_n}inv_B(w)=2^{n-1}n^2n!.
\end{equation*}
\end{thm}

\begin{proof}
Let $f~:~B_n \rightarrow B_n$ be a mapping that assigns every permutation onto its backward permutation, namely $f(w)=\overleftarrow{w}$. It is clear that $f$ is an involution. Thus we have $B_n=\{\overleftarrow{w}~:~w\in B_n\}$. Since $inv_B(w)+inv_B(\overleftarrow{w})= \binom{n}{2} +2\sum_{w_i<0}|w_i|$, we get 
\begin{align*}
2\sum_{w \in B_n}inv_B(w)&=\sum_{w \in B_n}inv_B(w)+\sum_{w \in B_n}inv_B(\overleftarrow{w})\\
&=\sum_{w \in B_n} \left( \binom{n}{2} +2\sum_{w_i<0}|w_i|  \right)\\
&=\sum_{w \in B_n} \binom{n}{2} +2 \sum_{w \in B_n} \left(  \sum_{w_i<0}|w_i| \right)\\
&=2^{n}n!  \binom{n}{2}+2 \sum_{i=1}^n  (2^{n-1}n!)i \\
&=2^{n}n!  \binom{n}{2}+n 2^{n-1}(n+1)! \\
&=2^{n}n^2n!
\end{align*}
as desired.
\end{proof}

\textbf{Questions:} \\
\textit{\textbf{1.} How can $q$-analogue of Mahonian numbers of type $B$ be defined in the sense of \cite{Ghemit2023}? As a consequence, in what form do the $q$-analogues of Theorem \ref{first} and \ref{second} appear?}\\
\textit{\textbf{2.} How can the structure of the generating function of $i_B(n,n-j)$ for a fixed non-negative integer $j$ be derived as an analogue of the equation (\ref{gf})?}\\
\textit{\textbf{3.} The Knuth-Netto \cite{Knuth1973, netto1901} formula for the $k$th Mahonian number $i(n,k)$ when $k\leq n$ has the form
$$i(n,k)=\binom{n+k-1}{k}+ \sum_{j=1}^{\infty}(-1)^j \binom{n+k-u_j-j-1}{k-u_j-j}+ \sum_{j=1}^{\infty}(-1)^j \binom{n+k-u_j-1}{k-u_j}$$
where $u_j=\frac{j(3j-1)}{2}$ is the $j$th pentagonal number. Thus, what exactly is the analogue of the Knuth-Netto formula for $i_B(n,k)$ when $k\leq n$?}

\end{document}